\documentclass[a4paper]{amsart}
\title[Real hypersurfaces with many prescribed components]{Existence of real algebraic hypersurfaces with many prescribed components}
\author{Michele Ancona}
\thanks{Institut de Recherche Math\'ematique Avanc\'ee, Universit\'e de Strasbourg.\\
 \emph{E-mail address}: \url{michele.ancona@math.unistra.fr}. }
\date{}

\usepackage[latin1]{inputenc}
\usepackage[T1]{fontenc}
\usepackage[english]{babel}
\usepackage{amsmath,amssymb,amsthm}
\usepackage{amsfonts}%
\usepackage{amssymb}%
\usepackage{indentfirst}
\usepackage{fancyhdr}
\usepackage{physics}
\usepackage{young}
\usepackage{graphicx}

\usepackage[vcentermath]{youngtab}

\RequirePackage[colorlinks,linkcolor=blue,citecolor=blue,urlcolor=blue]{hyperref}
\usepackage{enumerate}
\usepackage{color}
\usepackage{pst-fill,pst-grad,pst-plot,pst-eucl,pstricks-add,pst-node}
\usepackage{mathdots}
\usepackage{dsfont}
\usepackage{mathrsfs}

\theoremstyle{plain}
\newtheorem{thm}{Theorem}[section]
\newtheorem{lemma}[thm]{Lemma}
\newtheorem{prop}[thm]{Proposition}

\newtheorem{oss}[thm]{Remark}

\theoremstyle{definition}
\newtheorem{defn}[thm]{Definition}

\renewcommand{\P}{\mathbf{P}}
\newcommand{\E}{\mathbf{E}}

\newcommand{\Z}{\mathbf{Z}}
\newcommand{\R}{\mathbf{R}}
\newcommand{\C}{\mathbf{C}}

\newcommand{\Vol}{\textrm{Vol}}

\allowdisplaybreaks

\begin{document}
\maketitle
\begin{abstract} 
Given a real algebraic variety $X$ of dimension $n$, a very ample divisor $D$ on $X$ and a smooth closed hypersurface $\Sigma$ of $\R^n$, we construct real algebraic hypersurfaces in the linear system $\abs{mD}$ whose real locus contains many connected components diffeomorphic to $\Sigma$. 
As a consequence, we show the existence of real algebraic hypersurfaces in the linear system $\abs{mD}$ whose Betti numbers grow by the maximal order, as $m$ goes to infinity. As another application, we recover a result by D. Gayet on the existence of many disjoint lagrangians with prescribed topology in any smooth complex hypersurface of $\C\P^n$.
The results in the paper are  proved more generally for complete intersections.
 The proof of our main result uses probabilistic tools.
\end{abstract}
\section{Introduction}
This paper deals with the topology of real algebraic varieties. A real algebraic variety $X$ is an algebraic variety defined over $\R$.  We denote its real and complex locus respectively by $X(\R)$  and $X(\C)$ and its dimension by $n$. We assume that $X$ is smooth and projective and that its real locus is non empty.

\subsection{Topology of real algebraic hypersurfaces} A major restriction for the topology of real algebraic varieties is given by Smith-Thom inequality \cite{thomreelle}, which says that the sum of the $\Z/2$-Betti numbers of the real locus of a real algebraic variety is smaller or equal than the sum of the $\Z/2$-Betti numbers of its complex locus, that is 
\[
b\left(X(\R)\right):=\displaystyle\sum_{i=0}^{n}b_i\left(X(\R),\Z/2\right)\leq \sum_{i=0}^{2n}b_i\left(X(\R),\Z/2\right)=:b\left(X(\C)\right).
\]
In the case of equality, the variety is called a \emph{maximal variety}. For real algebraic curves, Smith-Thom inequality is known as Harnack-Klein inequality \cite{harnack,klein} and can be written $b_0\big(C(\R)\big)\leq g+1$, where $g$ is the  genus of $C(\C)$.

A fundamental question in real algebraic geometry asks whether a given real algebraic variety contains  maximal hypersurfaces in some given linear system. For example, for any $m\in\mathbf{N}$, there exists a degree $m$ maximal real algebraic curves in $\P^2$,  see \cite{harnack}.
However, in general, the answer to this question is negative. For example, given a non-maximal real algebraic curve $C$, the surface $C\times C$ does not contain any maximal curves.
Nevertheless, such surface contains \emph{asymptotically maximal hypersurfaces}. To give the definition of asymptotically maximal hypersurfaces, let us fix a very ample divisor $D$ of $X$ and consider the linear system $\abs{mD}, m>0$, consisting of real divisors linearly equivalent to $mD$. By Bertini theorem, a generic element $Z$ of $\abs{mD}$ is smooth and, by Ehresmann theorem, the topology of the complex locus $Z(\C)$ of $Z$ only depends on $D$ and on $m$. One can then compute the asymptotics $b\left(Z(\C)\right)=D^{n}m^n+O(m^{n-1})$, as $m\rightarrow\infty$, where $D^n$ denotes the top self-intersection number of $D$ (which is a positive integer) and $n$ is the dimension of $X$. By Smith-Thom inequality, we then have $$b\left(Z(\R)\right)\leq D^{n}m^n+O(m^{n-1})$$ and a sequence $Z_m\in\abs{mD}$ is called asymptotically maximal if $$\displaystyle\lim_{m\rightarrow\infty}\frac{b\left(Z_m(\R)\right)}{D^{n}m^n}\rightarrow 1.$$

Asymptotically maximal hypersurfaces exists in projective spaces \cite{itenbergviro}, in toric varieties \cite{bertrand} and in surfaces \cite[Theorem 5]{gwexp}. However, one does not know if any real algebraic variety contains asymptotically maximal hypersurfaces and it is not clear (at least for the author) if one should expect this property to be true for any $X$.

One of the consequence of the main result of this paper is  that \emph{any} real algebraic variety always contains real algebraic hypersurfaces whose Betti numbers grow as the maximal possible order. More precisely we obtain the following result.
\begin{thm}\label{thm existence} Let $X$ be a real algebraic variety of dimension $n$ and $D$ be a real very ample divisor of $X$.
 Then there exists $c>0$ and $m_0\in\mathbf{N}$ such that, for any $m\geq m_0$,  there exists $Z\in\abs{mD}$ with $b_i\left(Z(\R)\right)\geq cm^{n}$ for any $i\in\{0,\dots,n-1\}$.
\end{thm}
\subsection{Existence of hypersurfaces with prescribed components} We actually prove a more precise  result than Theorem \ref{thm existence}. Indeed, for any  smooth closed (not necessarly connected) hypersurface $\Sigma$ of $\R^n$, we can produce real algebraic hypersurfaces $Z_m\in\abs{mD}$ such that $Z_m(\R)$ contains at least $cm^n$ connected component diffeomorphic to $\Sigma$. This is the content of Theorem \ref{thm existence of components}, which is the main result of our paper. In order to state it, we need the following definition.
\begin{defn}\label{defN} Let $\Sigma$ be a closed smooth manifold, not necessarly connected.
For any smooth manifold $S$, we denote by $\mathcal{N}_{\Sigma}(S)$ the number of  connected components of $S$ that are diffeomorphic to $\Sigma$.
\end{defn}
\begin{thm}\label{thm existence of components} Let $X$ be a real algebraic variety of dimension $n$ and $D$ be a real very ample divisor of $X$. For any  smooth closed  hypersurface $\Sigma$ of $\R^n$,
  there exists $c>0$ and  $m_0\in\mathbf{N}$ such that, for any $m\geq m_0$,  there exists $Z\in\abs{mD}$ with $\mathcal{N}_{\Sigma}\left(Z(\R)\right)\geq cm^{n}$.
\end{thm}
Theorem \ref{thm existence} then follows from Theorem \ref{thm existence of components} by taking as $\Sigma$ any hypersurface with $b_i(\Sigma)\geq 1$ for any $i\in\{0,\dots,n-1\}$.

The proof of Theorem \ref{thm existence of components} is probabilistic and it is given in Section \ref{section probability}. The idea is to construct a probability measure on $\abs{mD}$ for which the expected value $\E\left(\mathcal{N}_{\Sigma}(Z(\R)) \right)$ of $\mathcal{N}_{\Sigma}(Z(\R))$, for $Z\in\abs{mD}$,  is at least $cm^n$, for $m$ large enough, where $c$ is a positive constant independent of $m$. This is the content of Theorem \ref{thm expected}.
Remark that in Section \ref{Section complete int} we will state a generalization of Theorem \ref{thm existence of components} for complete intersections.

\subsection{Existence of many lagrangians in complex hypersurfaces} We now give an application of Theorem \ref{thm existence of components} in symplectic geometry.
Recall that the complex projective space $\C\P^n$ is equipped with a natural symplectic form $\omega_{FS}$, called the Fubini-Study symplectic form. Any smooth complex hypersurface of $\C\P^n$ inherits, by restriction, a symplectic form. Recall also that a lagrangian  of a symplectic manifold $M$ is a smooth closed submanifold  of dimension half the dimension of $M$ and for which the restriction of the symplectic form is zero everywhere.
For example, $\R\P^n$ is a lagrangian of $\C\P^n$ and, more generally, the real locus of a real algebraic hypersurface of $\C\P^n$ is a lagrangian of its complex locus.
As an application of Theorem \ref{thm existence of components}, we obtain an easy proof of the following result by D. Gayet.
\begin{thm}\cite[Theorem 1.1]{GayetLagrangian}\label{lagr}
For any  smooth closed  hypersurface $\Sigma$ of $\R^n$, there exists $c>0$ and $m_0\in\mathbf{N}$ such that, for any $m\geq m_0$, any degree $m$ smooth complex hypersurface of $\C\P^n$ contains at least $cm^n$ pairwise disjoint lagrangians diffeomorphic to $\Sigma$.
\end{thm}
\begin{proof}
 Theorem \ref{thm existence of components} provides smooth complex hypersurfaces of degree $m$ with at least $cm^n$ pairwise disjoint lagrangians diffeomorphic to $\Sigma$ (as the real locus of a real algebraic hypersurface is lagrangian). As, from a symplectic point of view, all degree $m$ smooth complex hypersurfaces of $\C\P^n$ are isomorphic, this implies that  \textit{any} degree $m$ smooth complex hypersurface contains at least $cm^n$ pairwise disjoint lagrangians diffeomorphic to $\Sigma$.
 \end{proof}
 
Actually, the same result and argument  works for any K\"ahler $(X,\omega)$ equipped with an ample divisor $D$ with $\omega\in c_1(D)$, that \emph{can be defined over $\R$}.
Although our proof is simpler than the one in \cite{GayetLagrangian} (because we do not use any quantitative Moser-type argument), it should be noted that in \cite{GayetLagrangian} this result is proved for any complex projective variety $X$, and not just for those that can be defined over $\R$.

It is worth noticing that the order $m^n$ appearing in Theorem \ref{lagr} is optimal when the Euler characteristic of $\Sigma$ is nonzero. Indeed, as remarked in \cite[Corollary 1.2]{GayetLagrangian}, if $\chi(\Sigma)\neq 0$ and  $Z_m$ denotes  a smooth degree $m$ complex hypersurface of $\C\mathbf{P}^n$, then (the homology class of) the disjoint langrangians  diffeomorphic to $\Sigma$ are linearly independent in $H_{n-1}(Z_m,\mathbf{Z})$   and the dimension of $H_{n-1}(Z_m,\mathbf{Z})$ grows exactly as $m^n$ when $m\rightarrow\infty$.
\subsection{Existence of complete intersections with prescribed topology}\label{Section complete int} 
We can actually prove all the previous results not only for hypersurfaces, but more generally for complete intersections. For this, recall that, as for the case of hypersurfaces, the topology of the complex locus of a  complete intersection $Z_1\cap\dots\cap Z_r$, with $Z_i\in \abs{mD}$, does not depend on the choice of the hypersurfaces $Z_i$, if they are chosen generically. One can then compute the asymptotics for the total Betti number of $Z_1(\C)\cap\dots\cap Z_r(\C)$ for generic $Z_i\in\abs{mD}$ and get  
\[b(Z_1(\C)\cap\dots\cap Z_r(\C))=\binom{n-1}{r-1}D^nm^n+O(m^{n-1}) \] 
as $m\rightarrow\infty$, see \cite[Proposition 2.2]{anc6}.
In this setting, we have the following result.
 \begin{thm}\label{thm existence of components 2} Let $X$ be a real algebraic variety of dimension $n$ and $D$ be a real very ample divisor of $X$. For any $r$-codimensional  closed  submanifold $\Sigma$ of $\R^n$ with trivial normal bundle,
  there exists $c>0$ and $m_0\in\mathbf{N}$ such that for any $m\geq m_0$  there exists $Z_1,\dots,Z_r\in\abs{mD}$ with $$\mathcal{N}_{\Sigma}\left(Z_1(\R)\cap\dots\cap Z_r(\R)\right)\geq cm^{n}.$$
\end{thm}
Theorem \ref{thm existence of components} is exacly Theorem \ref{thm existence of components 2} for $r=1$. Remark indeed that any closed hypersurface of $\R^n$ has trivial normal bundle.
On the contrary, it is worth noting that Theorem \ref{thm existence of components 2} is not a consequence of Theorem \ref{thm existence of components} by some induction, for example by considering  $Z_1\cap Z_2$ as a hypersurface of $Z_1$ and then applying Theorem \ref{thm existence of components} to $Z_1$. Indeed the hypersurfaces $Z_1,Z_2\in\abs{mD}$  have degree  that varies with $m$, and thus no induction can be applied since Theorem \ref{thm existence of components} requires a fixed ambient variety. Even if one tried to get around this problem, fixed once for all one of the hypersurfaces, say $Z_1$, and considered $Z_1\cap Z_2$ as a hypersurface of  $Z_1$, then applying Theorem \ref{thm existence of components} to $Z_1$ one would  get only $cm^{n-1}$ connected components diffeomorphic to $\Sigma$, and not $cm^{n}$. 

\begin{oss} Analogues of Theorems \ref{thm existence} and \ref{lagr} can then be translated for complete intersections and are direct consequences of the Theorem \ref{thm existence of components 2}, in the same way that Theorems \ref{thm existence} and \ref{lagr} are consequences of the Theorem \ref{thm existence of components}.
\end{oss}
\subsection{Organization of the paper} The paper is organized as follows.\\ 
 In Section \ref{section probability}, we define a probability measure on the linear system $\abs{mD}$ and prove Theorem \ref{thm existence of components 2} admitting a result, namely Proposition \ref{barrier}. In Section \ref{section proof prop}, we prove Proposition \ref{barrier}.
 In Section \ref{section questions} we make some comments on this paper and on  related works and we adress some questions arising from this work.

\subsection*{Acknowledgments} This paper was written while I was postdoc at the Institut de Recherche Math\'ematique Avanc\'ee, Universit\'e de Strasbourg. I thank the IRMA for the excellent conditions given to me during my postdoc.
 
\section{Probability measures on the linear systems and proof of the main result}\label{section probability}
In this section we equip the linear system $\abs{mD}$ with a probability measure. Here $D$ is a very ample divisor of a real algebraic variety $X$.  With respect to this probability measure, we estimate the expected value of $\mathcal{N}_{\Sigma}(Z_1(\R)\cap\dots\cap Z_r(\R))$, for $Z_i\in\abs{mD}$. This is the content of the Theorem \ref{thm expected}, which directly implies Theorem \ref{thm existence of components 2}.

\subsection{Linear systems and hyperplane sections}\label{subsec hyperplane} The very ample divisor $D$ embeds $X$ into $\P^N$, for some $N\in\mathbf{N}$.  We fix once for all such an embedding and identify $X$ with its image in  $\P^N$. A real algebraic hypersurface $Z\in\abs{D}$ is then obtained as a intersection $H\cap X$ of $X$ with an hyperplane $H\in \abs{\mathcal{O}(1)}$ defined over $\R$. More generally, given a degree $m$ real algebraic hypersurfaces $H$ in $\P^N$, its intersection with $X$ defines a real algebraic hypersurface $Z:=H\cap X$ of $X$ in the linear system $\abs{mD}$. Note however that not all hypersurfaces lying in $\abs{mD}$ are of the form $H\cap X$, for  $H\in\abs{\mathcal{O}(m)}$, that is, the  map $H\in\abs{\mathcal{O}(m)}\mapsto H\cap X \in \abs{mD}$ is not surjective in general, when $m>1$.

\subsection{Probability measure on $\abs{mD}$} Let $\R^{hom}_m[X_0,\dots,X_N]$ be the space of real homogeneous polynomials of degree $m$ in $N+1$ variables. This space coincides with the space of global algebraic sections  of the line bundle $\mathcal{O}(m)$ on $\P^N$, which are defined over $\R$.
We endow the space $\R^{hom}_m[X_0,\dots,X_N]$ with a Gaussian probability measure $\mu_m$ as follows.  Let $h_m$ be the Fubini-Study metric on the line bundle $\mathcal{O}(m)$ and let $g_{FS}$ be the Riemannian Fubini-Study metric on $\P^N(\R)$. Then, for any pair of real polynomials $P_1,P_2\in \R^{hom}_m[X_0,\dots,X_N]$ we define 
\begin{equation}\label{L2 scalar product}
\langle P_1,P_2\rangle_{L^2}=\int_{\P^N(\R)}h_m(P_1,P_2)d\mathrm{vol}_{FS}.
\end{equation}
This scalar product induces a Gaussian measure $\mu_m$ on $\R^{hom}_m[X_0,\dots,X_N]$ defined by
\begin{equation}
\mu_m(U)=\frac{1}{\sqrt{\pi}^{d_m}}\int_{U}e^{-\norm{P}_{L^2}^2}dP
\label{gaussian measure}\end{equation}
for any open set $U\subset\R^{hom}_m[X_0,\dots,X_N]$, where $d_m=\dim \R^{hom}_m[X_0,\dots,X_N]$.
% In the literature, the probability space $(\R^{hom}_m[X_0,\dots,X_N],\mu_m)$ is sometimes known as the real Fubini-Study ensemble.

For any polynomial $P\in \R^{hom}_m[X_0,\dots,X_N]$, we denote by $V_P=\{P=0\}\subset \P^N$ the real algebraic hypersurface defined by $P$ and by $Z_P:=V_P\cap X$ the real algebraic hypersurface of $X$ defined by $P_{\mid X}$. As said Section \ref{subsec hyperplane},  the hypersurface $Z_P$ is an element of the linear system $\abs{mD}$.
\begin{defn}\label{defprob} The map $\varphi:P\in\R^{hom}_m[X_0,\dots,X_N]\mapsto Z_P\in \abs{mD}$ induces a probability measure $\mathrm{Prob}_m:=\varphi_*\mu_m$ on $\abs{mD}$. For any $r\in\{1,\dots,n\}$ we equip the $r$-fold cartesian products $\R^{hom}_m[X_0,\dots,X_N]^r$ and  $\abs{mD}^r$ with the product measures, that we still denote respectively by $\mu_m$ and $\mathrm{Prob}_m$. 
\end{defn}
\begin{oss}
By what we said in Section \ref{subsec hyperplane}, the probability measure $\mathrm{Prob}_m$ is degenerate, that is, it is supported in a proper subspace of $\abs{mD}^r$. 
\end{oss}
\subsection{Expected value of $\mathcal{N}_{\Sigma}(Z(\R))$} Let $\Sigma$ be a smooth closed   $r$-codimensional submanifold of $\R^n$ with trivial normal bundle. We are interested in the random variable $$(Z_1,\dots,Z_r)\in\abs{mD}^r\mapsto \mathcal{N}_{\Sigma}\left(Z_1(\R)\cap\dots\cap Z_r(\R)\right),$$ where $\mathcal{N}_{\Sigma}$ denotes the number of connected components that are diffeomorphic to $\Sigma$, see Definition \ref{defN}. The next theorem estimates from below the expectation 
\begin{multline*}
\E\left(\mathcal{N}_{\Sigma}(Z_1(\R)\cap\dots\cap Z_r(\R)) \right) \\
=\displaystyle\int_{(Z_1,\dots,Z_r)\in\abs{mD}^r}\mathcal{N}_{\Sigma}\left(Z_1(\R)\cap\dots\cap Z_r(\R)\right)d\mathrm{Prob}_m(Z)
\end{multline*}
where $\mathrm{Prob}_m$ is the probability measure on  $\abs{mD}^r$ defined in Definition \ref{defprob}.

\begin{thm}\label{thm expected} There exists $c>0$ and $m_0$ such that for any $m\geq m_0$  we have 
$$\E\left(\mathcal{N}_{\Sigma}(Z_1(\R)\cap\dots\cap Z_r(\R)) \right)\geq cm^n.$$
\end{thm}
Theorem \ref{thm expected} directly implies Theorem \ref{thm existence of components 2}, since we have the trivial  bound $$\sup_{(Z_1,\dots,Z_r)\in\abs{mD}^r}\mathcal{N}_{\Sigma}\big(Z_1(\R)\cap\dots\cap Z_r(\R)\big)\geq \E\left(\mathcal{N}_{\Sigma}(Z_1(\R)\cap\dots\cap Z_r(\R)) \right).$$

In the rest of the section we prove Theorem \ref{thm expected}. For this, remark that the real locus $X(\R)$ of $X$ inherits from $\P^N(\R)$ a Riemannian metric $g:=g_{FS \mid X(\R)}$. For any point $x\in X(\R)$, we denote by $B(x,R)$ the geodesic ball of radius $R$ around $x$.

Theorem \ref{thm expected} will be a direct consequence of the following proposition.
\begin{prop}\label{barrier} There exists $R,c_\Sigma>0$ and $m_0$ such that for any point $x\in X(\R)$ and any $m\geq m_0$ we have 
\[\mathrm{Prob}_m\bigg\{(Z_1,\dots,Z_r)\in\abs{mD}^r,\hspace{1mm} \mathcal{N}_{\Sigma}\left(Z_1(\R)\cap\dots\cap Z_r(\R)\cap B\big(x,\frac{R}{m}\big)\right)\geq 1\bigg\}\geq c_\Sigma.\]
\end{prop}
We postpone the proof of this proposition to the next section. We now show how this implies Theorem \ref{thm expected}.
\begin{proof}[Proof of Theorem \ref{thm expected}]
Let $R$ be the positive constant given by Proposition \ref{barrier}. For any $m\in\mathbf{N}$, we fix a set of points $\{x_i\}_{i\in I_m}$ of $X(\R)$ with the property that the balls $B(x_i,\frac{R}{m})$ are all disjoint and the balls $B\left(x_i,\frac{2R}{m}\right)$ cover the whole $X(\R)$. In particular, we get $$\Vol\left(X(\R)\right)\leq \sum_{i\in I_m}\Vol\left(B\left(x_i,\frac{2R}{m}\right)\right)\leq \abs{I_m}\Vol\left(B_{\R^n}\left(0,\frac{3R}{m}\right)\right)$$ for $m$ large enough. This implies that for $m$ large enough we have
\begin{equation}\label{estimates on volumes}
\abs{I_m}\geq \frac{\Vol\left(X(\R)\right)}{\Vol\big(B_{\R^n}\left(0,\frac{3R}{m}\right)\big)}=c'\Vol\left(X(\R)\right)m^n.
\end{equation}
We can now estimate the expected value of $\mathcal{N}_{\Sigma}(Z_1(\R)\cap\dots\cap Z_r(\R)) $, for $Z\in\abs{mD}$. We have the lower bound
\begin{multline}
\E\left(\mathcal{N}_{\Sigma}(Z_1(\R)\cap\dots\cap Z_r(\R)) \right)\geq \\
 \sum_{i\in I_m}\mathrm{Prob}_m\bigg\{(Z_1,\dots,Z_r)\in\abs{mD}^r,\hspace{1mm} \mathcal{N}_{\Sigma}\left(Z_1(\R)\cap\dots\cap Z_r(\R)\cap B\big(x_i,\frac{R}{m}\big)\right)\geq 1\bigg\}.
\end{multline}
By Proposition \ref{barrier}, the latter is bigger or equal than $c_\Sigma\abs{I_m}$ which, by \eqref{estimates on volumes}, is bigger or equal than $c'\Vol\left(X(\R)\right)m^n=:cm^n$. Hence the result.
\end{proof}
\section{Proof of Proposition \ref{barrier}}\label{section proof prop}
In this section we prove Proposition \ref{barrier}. Following \cite{gw2},
the idea is to find  one complete intersection $Z_1\cap\cdots\cap Z_r$ of degree $m\gg 1$ such that $$\mathcal{N}_{\Sigma}\left(Z_1(\R)\cap\dots\cap Z_r(\R)\cap B\big(x,\frac{R}{m}\big)\right)\geq 1$$ and then prove that this complete intersection can  perturbed  in a precise quantitative way so that this property happens with positive probability. 
Actually, still keeping in mind this idea,  we will  use a slightly different, more flexible, point of view,  developped recently in \cite{LerarioStecconi}. 

The formalism developped in \cite{LerarioStecconi} works better when one considers functions rather than hypersurfaces or sections of line bundles.   We then  consider $S^N$  the double covering  of $\P^N(\R)$ and $M\subset S^N$ the double covering of $X(\R)$. We fix on $S^N$ the standard round metric $g_{S^N}$ (so that the Fubini-Study metric $g_{FS}$ is the quotient of $g_{S^N}$ by the antipodal map).
For any  $r$-tuple of real  polynomials $\underline{P}=(P_1,\dots,P_r)\in \R^{hom}_m[X_0,\dots,X_N]^r$ we denote by $V_{\underline{P}}$ the complete intersection in $\P^N$ defined by $\{P_1=0\}\cap\dots\cap \{P_r=0\}$ and by $Z_{\underline{P}}$
 the complete intersection in $X$ defined by $V_{\underline{P}}\cap X$.
Fix any point $x\in M$. To prove Proposition \ref{barrier}, it is sufficient to show that 
\begin{equation}\label{new estimate}
\mu_m\left\{\underline{P}\in \R^{hom}_m[X_0,\dots,X_N]^r, \mathcal{N}_{\Sigma}\left(Z_{\underline{P}}(\R)\cap B\big(x,\frac{R}{m}\big)\right)\geq 1\right\}\geq c_\Sigma.
\end{equation}
where the probability measure $\mu_m$ is defined in \eqref{gaussian measure}.

Remark that the scalar product \eqref{L2 scalar product} can be also defined by 
$$\langle P_1,P_2\rangle_{L^2}=\int_{S^N}P_1P_2d\mathrm{vol}_{S^N}.$$
A random   polynomial $P\in \R^{hom}_m[X_0,\dots,X_N]$ with respect to the Gaussian measure $\mu_m$ can be then written as 
$$P=\sum_{i=1}^{d_N}a_iP_i$$
where $d_N=\dim \R^{hom}_m[X_0,\dots,X_N]$, the family $\{P_i\}_{i}$ is an orthonormal basis of $\R^{hom}_m[X_0,\dots,X_N]$ and $(a_i)_i$ are independent standard Gaussian variables. A random $\underline{P}\in \R^{hom}_m[X_0,\dots,X_N]^r$ is then a $r$-tuple $\underline{P}=(P_1,\dots,P_r)$  of independent random polynomials $P_i$.

The covariance function $\mathcal{K}_m(x,y):=\E(P(x)P(y))$ of the random polynomial $P$ is known to have a universal limit at scale $m^{-1}$ (see \cite[Theorem 4.4]{hormander}). This means that for any fixed $R>0$, any point $x$ and  any $u,v\in B_{\R^N}(0,R)\subset \R^N\simeq T_xS^N$ (where $T_xS^N\simeq \R^N$ is given by any isometry) one has that  $$K_{m,x}(u,v):=m^{-n}\mathcal{K}_m\left(\exp_x\left(\frac{u}{m}\right),\exp_x\left(\frac{v}{m}\right)\right)$$ converges in the $\mathscr{C}^{\infty}$-topology to a function $K:B_{\R^N}(0,R)\times B_{\R^N}(0,R)\rightarrow\R$ which is independent of $x$. 
The function $K$ is explicit and given by 
$$K(u,v)=\int_{B_{\R^N}(0,1)}e^{i\langle u-v,\xi\rangle}d\xi.$$ 

By restriction, we then have that the covariance function of the random polynomial $P$ \textit{ restricted to $M$} also has a universal limit at scale $m^{-1}$.  This universal limit is nothing but the restriction of $K$ to $B_{\R^n}(0,R)\times B_{\R^n}(0,R)$, where we see $T_xM\subset T_xS^N$ as the subspace $\R^n\subset \R^N$ given by the last $N-n$ coordinates equal to $0$.

The restriction of the function $K$ to $B_{\R^n}(0,R)\times B_{\R^n}(0,R)$ is the covariance function of a Gaussian field $F:B_{\R^n}(0,R)\rightarrow\R$ (such Gaussian field is precisely characterized by the property $K(u,v)=\mathbf{E}(F(u)F(v))$).

Consider $r$ independent copies $F_1,\dots, F_r$ of the Gaussian field $F$, and denote by $\underline{F}=(F_1,\dots,F_r)$, which is a Gaussian field taking values in $\mathscr{C}^{\infty}\left(B_{\R^n}(0,R),\R^r\right)$.
We now need three lemmas about the geometry of this Gaussian field $\underline{F}$.

\begin{lemma}\label{lemma converg} The restriction to $B_{\R^n}(0,R)$ of the rescaled $r$-tuple of random polynomials $$\underline{P_m}(\cdot):=\left(m^{-n/2}P_1\left(\exp_x\left(\frac{\cdot}{m}\right)\right),\dots,m^{-n/2}P_r\left(\exp\left(\frac{\cdot}{m}\right)\right)\right)$$  converges in probability to $\underline{F}$  as $m\rightarrow\infty$.
\end{lemma}
\begin{proof}
By construction, we have that the covariance function of the rescaled random polynomial $m^{-n/2}P\left(\exp_x\left(\frac{\cdot}{m}\right)\right)$ converges (in the $\mathscr{C}^{\infty}$-topology) to the covariance function of $F$. This implies that the covariance function of $\underline{P_m}$ converges to the covariance function of $\underline{F}$. By \cite[Theorem 5]{LerarioStecconi}, this implies the result.
\end{proof}
\begin{lemma}\label{support} For any smooth function $f\in\mathscr{C}^{\infty}\left(\bar{B}_{\R^n}(0,R),\R^r\right)$ and any neighborhood $U$ of $f$ in $\mathscr{C}^{\infty}\left(\bar{B}_{\R^n}(0,R),\R^r\right)$, the probability that $\underline{F}$ takes value in $U$ is strictly positive. That is,  the support of $\underline{F}$ contains the space $\mathscr{C}^{\infty}\left(\bar{B}_{\R^n}(0,R),\R^r\right)$. Here, $\bar{B}_{\R^n}(0,R)$ denotes the closed ball of radius $R$ around $0$.
\end{lemma}
\begin{proof}
It is enough to prove that the support of each components of $\underline{F}$  includes $\mathscr{C}^{\infty}\left(\bar{B}_{\R^n}(0,R),\R\right)$. By the density of polynomials inside  $\mathscr{C}^{\infty}\left(\bar{B}_{\R^n}(0,R),\R\right)$, it is actually enough to prove that any monomial $x_1^{k_1}\cdots x_n^{k_n}$ is in the support  of each components of $\underline{F}$.

By definition, a component of $\underline{F}$ is a copy of the Gaussian field $F$ whose covariance function is $$K(u,v)=\int_{B_{\R^N}(0,1)}e^{i\langle u-v,\xi\rangle}d\xi$$ 
and whose spectral measure $\sigma$ is the Lebesgue measure on $B_{\R^N}(0,1)$.

The support of $F$ is known to be equal to the closure (in the $\mathscr{C}^{\infty}(\bar{B}(0,R),\R)$ topology) of the following space of function
$$\mathcal{H}_F=\mathrm{Span}\{K_v, v\in\R^n\}$$
where $K_v(x):=\int_{B_{\R^N}(0,1)}e^{i\langle x-v,\xi\rangle}d\xi$, see \cite[Theorem 6]{LerarioStecconi}.
It is then enough to prove that any monomial $x_1^{k_1}\cdots x_n^{k_n}$ can be approximated (uniformly together with its derivatives) by elements of the space $\mathcal{H}_F$.

In order to prove this, let $\varphi:\R^N\rightarrow [0,1]\subset\R$ be an even smooth function, which equals $0$ outside $B_{\R^N}(0,1)$ and such that $\int_{\R^N}\varphi(\xi)\xi=1$.
Let us consider the function $\varphi_t(\xi)=\frac{1}{t^N}\varphi(\frac{\xi}{t})$, for any $0<t\leq 1$.
  Given $k_1,\dots,k_n\in\mathbf{N}$, with $\sum_{i=1}^nk_i=k$, consider the function $$\frac{\partial^k}{\partial \xi_1^{k_1}\cdots\partial \xi_n^{k_n}}\varphi_t.$$
 By construction, the value at $x\in B_{\R^n}(0,R)$ of its inverse Fourier transform equals
 $$\int_{\R^N}\frac{\partial^k}{\partial \xi_1^{k_1}\cdots\partial \xi_n^{k_n}}\varphi_t(\xi)e^{i\langle \xi,x\rangle}d\xi=(-1)^kx_1^{k_1}\cdots x_n^{k_n}\int_{\R^N}\varphi_t(\xi)e^{i\langle \xi,x\rangle}d\xi$$
where $\int_{\R^N}\varphi_t(\xi)e^{-i\langle \xi,x\rangle}d\xi\xrightarrow{t\rightarrow 0} 1$ in $\mathscr{C}^\infty(\bar{B}_{\R^n}(0,R),\R)$.
Remark also that the support of $\varphi_t$ is included in $\bar{B}_{\R^N}(0,1)$ for any $t\leq 1$. Denote by $f_t$ the inverse Fourier transform of $(-1)^k\frac{\partial^k}{\partial \xi_1^{k_1}\cdots\partial \xi_n^{k_n}}\varphi_t.$ We then have that the function 
$$g_t(x):=\int_{v\in\R^N}\int_{\xi\in\bar{B}_{\R^N}(0,1)}f_t(v)e^{i\langle x-v,\xi\rangle}dvd\xi$$
converges to $x_1^{k_1}\cdots x_n^{k_n}$ in the $\mathscr{C}^\infty(\bar{B}_{\R^n}(0,R),\R)$ topology  as $t\rightarrow 0$. On the other hand, the functions $g_t$ can be approximated in the  $\mathscr{C}^\infty(\bar{B}_{\R^n}(0,R),\R)$ topology by elements of $\mathcal{H}_F$, by taking Riemann sums.
  This implies that the monomial $x_1^{k_1}\cdots x_n^{k_n}$ can be approximated (uniformly together with its derivatives) by elements of $\mathcal{H}_F$. Hence the result.
 \end{proof}

\begin{lemma}\label{stability}
The probability that the zero  locus of $\underline{F}$ is diffeomorphic to $\Sigma$ is bigger than some strictly positive number $c_\Sigma$. 
\end{lemma}
\begin{proof}
Let  $f_{\Sigma}\in \mathscr{C}^{\infty}\left(\bar{B}_{\R^n}(0,R),\R^r\right)$  be a function vanishing transversally and such that $\{f_\Sigma=0\}$ is diffeomorphic to $\Sigma$. Remark that such function exists exactly by the hypothesis that the normal bundle of $\Sigma$ is trivial. 
By Thom isotopy lemma,  the zero set of any other function $f$ lying in a small neighborhood $U_\Sigma$ of $f_\Sigma$ is still diffeomorphic to $\Sigma$. 
By Lemma \ref{support}, $f_{\Sigma}$ is in the support of $\underline{F}$, and then the probability that $\underline{F}$ takes value in $U_\Sigma$ is strictly positive, which proves the lemma.
\end{proof}
We are now able to finish the proof of Proposition \ref{barrier}. 
By Lemmas \ref{lemma converg} and \ref{stability} and by {\cite[Theorem 5]{LerarioStecconi},} we  obtain that, for $m$ large enough, the probability that $Z_{\underline{P}}(\R)\cap B(x,\frac{R}{m})$, for $\underline{P}\in\R_m^{hom}[X_0,\dots,X_N]^r$, has a connected component diffeomorphic to $\Sigma$ is bigger than $c_\Sigma$. This is exactly \eqref{new estimate}, which finishes the proof of Proposition \ref{barrier}. \qed
\section{Further comments and related results}\label{section questions}

\subsection{Some questions} Let $D$ be a divisor on a real algebraic variety of dimension $n$. For any $i\in\{0,\dots,n-1\}$ we define
$$v_i(D)=\displaystyle\limsup_{m\rightarrow\infty}\frac{1}{m^n}\sup_{Z\in\abs{mD}}b_i(Z(\R))\hspace{4mm}\mathrm{and}\hspace{4mm}v(D)=\displaystyle\limsup_{m\rightarrow\infty}\frac{1}{m^n}\sup_{Z\in\abs{mD}}b(Z(\R)).$$
One always has  the trivial inequality $v_i(D)\leq v(D)$. The present paper implies that $v_i(D)>0$ if $D$ is ample. In this case, the existence of asymptotically maximal hypersurfaces in $\abs{mD}, m\gg 0$ is equivalent to the equality $v(D)=D^n$, where $D^n$  denotes the top self-intersection number of $D$. 
What about the other numbers $v_i(D)$? Are there examples of $(X,D)$ for which $v_i(D)<v(D)$, for some $i$? Does there exist relations between the numbers $v_i(D)$, apart from the trivial equality $v_i(D)=v_{n-i-1}(D)$? 
In real algebraic geometry, these numbers may refine the concept of volume of divisors. Recall that, if $D$ is a divisor of a complex manifold $X$ then the volume of $D$ is defined by the formula
$$\Vol(D)=\displaystyle\limsup_{m\rightarrow\infty}\frac{n!}{m^n}\dim H^0(X,\mathcal{O}(mD)).$$
The volume of a divisor encodes some informations about it. For example if $D$ is ample, then $\Vol(D)=D^n$. A divisor $D$ is big precisely when $\Vol(D)>0$. If now we consider $X$ and $D$ to be defined over $\R$, can one caracterize the bigness of $D$ in terms of $v_i(D)$? If $D$ is a big, are the numbers $v_i(D)$ positive for any $i\in\{0,\dots,n-1\}$?

\subsection{Some comments on the probability measures on $\abs{mD}$}
Let $D$ be an ample divisor over a real algebraic variety $X$ and denote by $H^0(X,\mathcal{O}(mD))$ the space over global algebraic sections of $\mathcal{O}(mD)$ defined over $\R$.
From the algebro-geometric point of view, the probability measures on $H^0(X,\mathcal{O}(mD))$ that seem more natural are the so-called complex Fubini-Study measures. These  measures are the Gaussian measure on $H^0(X,\mathcal{O}(mD))$ induced by the scalar product defined by 
\begin{equation}\label{complex scalar product}
\langle s_1,s_2\rangle=\int_{X(\C)}h^{\otimes m}(s_1,s_2)\frac{\omega^{\wedge n}}{n!}
\end{equation}
where $h$ is a hermitian metric on $\mathcal{O}(D)$ with positive curvature $\omega$.
It should be noted that  the Kodaira embeddings $\Phi_m:X(\C)\rightarrow \P^{d_m}(\C)$ induced by the Hermitian products \eqref{complex scalar product} are asymptotically isometries \cite{tian,bouche}, in the sense that $\frac{1}{m}\Phi^*_m\omega_{FS}\rightarrow\omega$, as $m\rightarrow\infty$.
With respect to the probability measure induced by the Hermitian product \eqref{complex scalar product},  asymptotically maximal hypersurfaces  are exponentially rare \cite{anc5,diatta,gwexp}. This reflects the fact that it is very difficult in practice to construct such hypersurfaces: if one randomly writes down an equation, it is unlikely to define a (almost) maximal hypersurface.
For these probability measures, the expected value $\E\left(b_i\left(Z(\R)\right)\right)$ of each Betti number of a random real hypersurface $Z\in\abs{mD}$ is of order $m^{n/2}$ (see \cite{gw3,gw2}) and any closed hypersurface $\Sigma\subset\R^n$ appears with positive probability in any ball of radius $\sim m^{-1/2}$ (see \cite{gw2}).

In contrast, the probability measure used in this paper seems less natural from an algebraic geometry point of view and produces, because of this, more extreme effects: the average of the Betti numbers of a random hypersurface is of maximal order (for $X=\P^n$ and for the number of connected components, this was proved in \cite{ll,ns}) and this implies precisely the existence of  hypersurfaces with rich topologies. Note that for the sphere $S^n$ the existence of real algebraic hypersurfaces with many components diffeomorphic to a given closed hypersurface $\Sigma\subset\R^n$ can also be obtain from \cite{gwuniversalcomponent}, by remarking that linear combinations of the Laplacian eigenfunctions  with eigenvalues smaller than $L$ coincide with homogenous polynomials of degree $\sim \sqrt{L}$.

We recall that for the projective space, and more generally for toric varieties, the existence of asymptotically maximal hypersurfaces is known. In these cases, where more combinatorial tools are accessible, it would be interesting to construct probability measures in which such hypersurfaces appear with high probability.

\bibliographystyle{plain}
\bibliography{biblio}
\end{document}